\documentclass[12pt]{article}
\usepackage{e-jc}

\usepackage{amsthm,amsmath,amssymb}

\usepackage{graphicx}

\usepackage[colorlinks=true,citecolor=black,linkcolor=black,urlcolor=blue]{hyperref}

\newcommand{\arxiv}[1]{\href{http://arxiv.org/abs/#1}{\texttt{arXiv:#1}}}

\usepackage{enumitem}

\theoremstyle{plain}
\newtheorem{theorem}{Theorem}
\newtheorem{lemma}[theorem]{Lemma}
\newtheorem{corollary}[theorem]{Corollary}

\theoremstyle{definition}

\newtheorem{notation}[theorem]{Notation}
\newtheorem{convention}[theorem]{Convention}

\theoremstyle{remark}

\usepackage{mathtools}
\usepackage{blkarray}
\usepackage[table]{xcolor}

\def\zo {$0\text{-}1$\ }

\def\N{\mathbb N}

\def\stirling#1#2{\genfrac{\{}{\}}{0pt}{}{#1}{#2}}

\title{\bf Bijective enumerations of\\$\bold\Gamma$-free 0-1 matrices}

\author{Be\'ata B\'enyi\\
\small Faculty of Water Sciences, National University of Public Service, Hungary\\
\small\tt beata.benyi@gmail.com\\
\and
G\'abor V.~Nagy\\
\small Bolyai Institute, University of Szeged, Hungary\\
\small\tt ngaba@math.u-szeged.hu}

\date{\today}

\begin{document}

\maketitle


\begin{abstract}
We construct a new bijection between the set of $n\times k$ $0$-$1$ matrices with no three $1$'s forming a $\Gamma$ configuration and the set of $(n,k)$-Callan sequences, a simple structure counted by poly-Bernoulli numbers.
We give two applications of this result: We derive the generating function of $\Gamma$-free matrices,
and we give a new bijective proof for an elegant result of Aval et al.\ that states that the number of complete non-ambiguous forests with $n$ leaves
is equal to the number of pairs of permutations of $\{1,\dots,n\}$ with no common rise.
\end{abstract}

\section{Introduction}\label{sec1}

We call a $0$-$1$ matrix \emph{$\Gamma$-free} if it does not contain $1$'s in positions such that they form a $\Gamma$ configuration; i.e.\ two $1$'s in the same row and a third $1$ below the left of these in the same column.
$\Gamma=\begin{smallmatrix}1 &1\\1& *\end{smallmatrix}$. For instance, matrix $A$ is not $\Gamma$-free because the bold $1$'s form a $\Gamma$ configuration, while matrix $B$ is a $\Gamma$-free matrix.
\begin{align*}
A=\begin{pmatrix}
0 &{\bf 1}&{\bf 1}&0\\1 & 1& 0& 0\\0&{\bf 1}&0&1
\end{pmatrix}, \quad
B=\begin{pmatrix}
0 & 1 & 0 & 0\\0 & 1& 1 & 0\\ 1& 0& 1& 1
\end{pmatrix}
\end{align*}
Clearly, we can say that a matrix is $\Gamma$-free if and only if it does not contain any of the submatrices from the following set:
\begin{align*}
\left\{\begin{pmatrix}1&1\\1&0\end{pmatrix},
\begin{pmatrix}1&1\\1&1\end{pmatrix}
\right\}.
\end{align*}
Pattern avoidance is an important notion in combinatorics. Matrices were also investigated from different point of view in this context; both extremal \cite{FH} and enumerative \cite{JuSeo},\cite{Kitaev} results are known.\par
$\Gamma$-free $0$-$1$ matrices of size $n\times k$ contain at most $n+k-1$ $1$'s \cite{FH}.
The set of $n\times k$ $0$-$1$ $\Gamma$-free matrices is one of the matrix classes that are enumerated by the poly-Bernoulli numbers, $B_n^{(-k)}$ \cite{BH1}.
Besides matrix classes that are characterized by excluded submatrices there are several other combinatorial objects that are enumerated by the poly-Bernoulli numbers.
For instance, permutations with a given exceedance set, permutations with a constraint on the distance of their values and images, Callan permutations, acyclic orientations of complete bipartite graphs, non-ambiguous forests, etc.
For further details, including recurrence relations and the original definition of poly-Bernoulli numbers via generating function, see \cite{BH1}, \cite{BH2} and \cite{Brewbaker}.
There is also a nice combinatorial formula of the poly-Bernoulli numbers of negative $k$ indices: For $k>0$,
\begin{equation}\label{formula}
B_{n}^{(-k)} =  \sum_{m=0}^{\min(n,k)}m!\stirling{n+1}{m+1}m!\stirling{k+1}{m+1},
\end{equation}
where $\stirling{n}{m}$ denotes a Stirling number of the second kind.
Table~\ref{tablazat} shows the values of $B_n^{(-k)}$ for small $k$ and $n$:
\begin{table}[h]\label{tablazat}%
\begin{center}%
\begin{tabular}{|c||c|c|c|c|c|c|}%
\hline
$n$, $k$
  & 0 & 1 & 2 & 3 & 4 & 5\\
\hline\hline
0 & 1 & 1& 1& 1 & 1 & 1\\
\hline
1 & 1 & 2 & 4 & 8 & 16 & 32\\
\hline
2 & 1 & 4 & 14 & 46 & 146 & 454 \\
\hline
3 & 1 & 8 & 46 & 230 & 1066 & 4718 \\
\hline
4 & 1 & 16 & 146 & 1066 & 6906 & 41506\\
\hline
5 & 1 & 32 & 454 & 4718 & 41506 & 329462\\
\hline%
\end{tabular}%
\caption {The poly-Bernoulli numbers $B_n^{(-k)}$.}%
\end{center}%
\end{table}\par
From \eqref{formula}, we give an obvious combinatorial interpretation of the numbers $B_n^{(-k)}$, which will be regarded as their combinatorial definition in this paper.
(This interpretation is essentially the same as the one that counts Callan permutations.)
On an \emph{$(n,k)$-Callan sequence} we mean a sequence $(S_1,T_1),\dots,(S_m,T_m)$ for some $m\in\N_0$ such that $S_1,\dots,S_m$ are pairwise disjoint nonempty subsets of $\{1,\dots,n\}$,
and $T_1,\dots,T_m$ are pairwise disjoint nonempty subsets of $\{1,\dots,k\}$. We note that the empty sequence is also a Callan sequence with $m=0$.
\begin{lemma}\label{polyB}
For $k>0$, $B_n^{(-k)}$ counts the number of $(n,k)$-Callan sequences.
\end{lemma}
\begin{proof}
For a fixed length $m$, there are $m!\stirling{n+1}{m+1}$ ways to give the sequence $S_1,\dots,S_m$. This is because if we partition $\{1,\dots,n+1\}$ into $m+1$ (nonempty) classes,
and order the $m$ classes not containing the element $n+1$ arbitrarily, then we obtain every possible $S_1,\dots, S_m$ sequences uniquely in this way.
($n+1$ can be thought as a ``dummy'' element, which is introduced to identify the class containing the elements of $\{1,\dots,n\}\setminus\cup_{i=1}^m S_i$ and to allow this set to be empty.)
Analogously, there are $m!\stirling{k+1}{m+1}$ ways to give the sequence $T_1,\dots,T_m$.
The sequences $(S_i)_{i=1}^m$ and $(T_i)_{i=1}^m$ are independent from each other and, trivially, $m$ can be at most $\min(n,k)$, so the statement follows from formula~\eqref{formula}.
\end{proof}\par
Now we can state our first main theorem.
\begin{theorem}\label{tetel1}
There exists a bijection between the set of $\Gamma$-free $n\times k$ \zo matrices
and the set of $(n,k)$-Callan sequences. Thus the number of $\Gamma$-free $n\times k$ \zo matrices is $B_n^{(-k)}$.
\end{theorem}\par
As mentioned earlier, the $\Gamma$-free $0$-$1$ matrices have already been enumerated in \cite{BH1}. In that paper, B\'enyi and Hajnal take an obvious combinatorial interpretation of poly-Bernoulli numbers (which is basically the same as ours, Lemma~\ref{polyB})
and give a bijective proof for the second statement of Theorem~\ref{tetel1}. Their proof involved a lot of technical details and it was still desirable to find a simple bijective explanation for this result;
a direct simple bijection that exhibits the connection between $\Gamma$-free matrices and Callan sequences.
In Section~\ref{sec2} we define such a bijection, which is essentially different from the one in \cite{BH1}.
It reveals the inner structure of $\Gamma$-free matrices from a new point of view. As an application, we derive the generating function of $\Gamma$-free matrices in Section~\ref{sec2.5}.
Furthermore, we use our bijection in Section~\ref{sec3} to prove bijectively that the set of complete non-ambiguous forests and the set of permutation pairs with no common rise are equinumeruous.\par
$\Gamma$-free matrices are very closely related to non-ambiguous trees and forests that were introduced in \cite{Aval1}.
A \emph{non-ambiguous tree} of size $n$ is a set $A$ of $n$ points: $v=(x(v),y(v))\in \N^+\times \N^+$ satisfying the following conditions:
\begin{enumerate}
\item $(1,1)\in A$ is the root of the tree;
\item for a given non-root point $p\in A$, there exists one point $q\in A$ such that $y(q)<y(p)$ and $x(q)=x(p)$ or one point $s$ such that $x(s)<x(p)$ and $y(s)=y(p)$ but not both;
\item there is no empty line between two given points: if there exists a point $p\in A$ such that $x(p)=x$ (resp.\ $y(p)=y$), then for every $x'<x$ (resp.\ $y'<y$) there exists a point $q\in A$ such that $x(q)=x'$ (resp.\ $y(q)=y'$).
\end{enumerate}
This structure can be viewed as a rooted binary tree graph on vertex set $A$ with root $(1,1)$ in which the parent of a non-root vertex $p$ is the nearest $q$ or $s$ from condition~2.
The name ``non-ambiguous'' comes from the property that the parent of $p$ can be uniquely recovered from the vertex set $A$, since either $q$ or $s$ does not exist.
In order to be consistent with the notion of $\Gamma$-free matrices, we slightly modified the original definition of non-ambiguous trees (we translated $A$), and we follow an unusual convention in this paper:
\begin{convention}\label{konv}
Throughout this paper, the rows of a matrix
are always indexed {\it from bottom to top\/} and the columns are indexed {\it from right to left\/}.
\end{convention}
The \emph{characteristic matrix} of a finite set $A\subseteq \N^+\times \N^+$ is a $0$-$1$ matrix $\chi_A$ with $\max_{v\in A} x(v)$ rows and $\max_{v\in A} y(v)$ columns, such that, \emph{following Convention~\ref{konv}},
there is a $1$ in position $(i,j)$ of $\chi_A$, if and only if $(i,j)\in A$.
Clearly, the characteristic matrix of a non-ambiguous tree is $\Gamma$-free by condition~2, so non-ambiguous trees can be thought as special $\Gamma$-free $0$-$1$ matrices.
It turns out that there is a graph theoretic terminology for $\Gamma$-free $0$-$1$ matrices in literature.
A \emph{non-ambiguous forest} is a finite set $A\subseteq \N^+\times \N^+$ such that $\chi_A$ is a $\Gamma$-free matrix without all-$0$ rows and columns.
This is actually the original definition given in \cite{Aval1}, but alternatively we could say that $A$ is a non-ambiguous forest, if it satisfies condition~3 in the definition of non-ambiguous trees and
the modified condition~2 that allows the possibility that neither such a $q$ nor such an $s$ exists.
Analogously to trees, a non-ambiguous forest has an underlying (rooted) binary forest structure, see Figure~\ref{fig0}. We note that non-ambiguous trees are exactly those non-ambiguous forests that have one component.
We say that a non-ambiguous tree or forest is \emph{complete} if its vertices have either $0$ or $2$ children.\par
\begin{figure}[h]%
\includegraphics[scale=0.42]{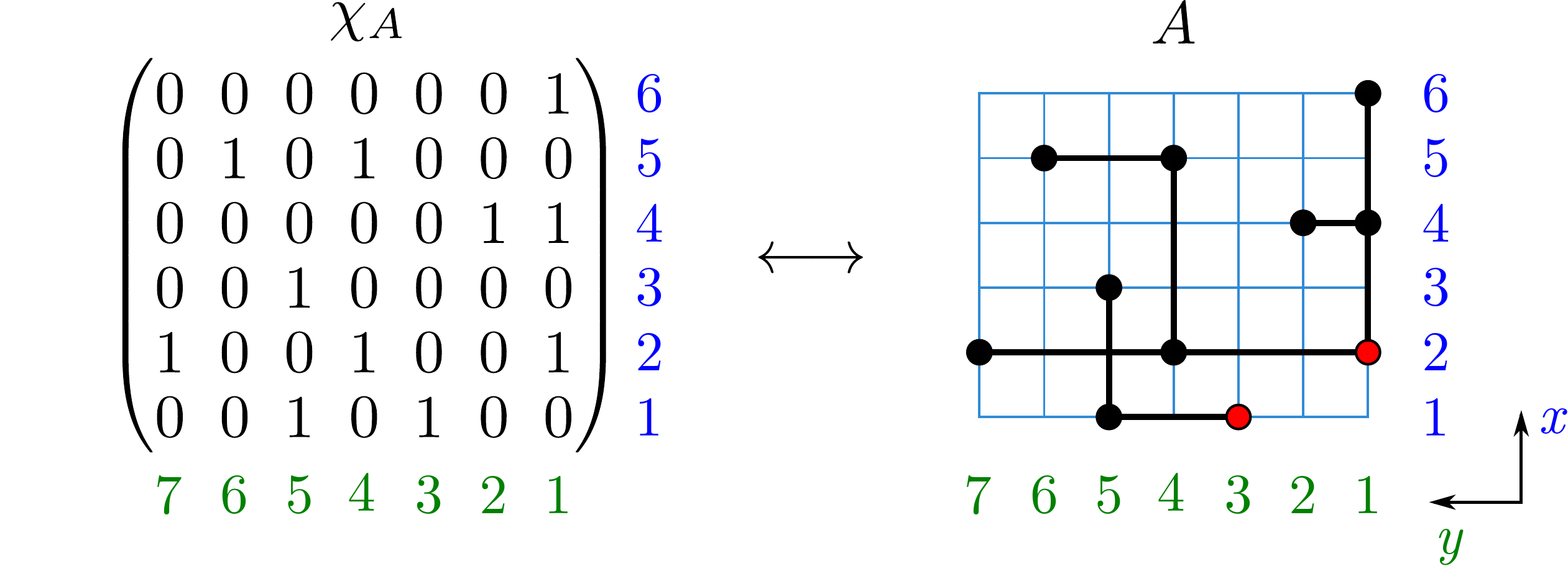}%
\centering
\caption{A $\Gamma$-free $0$-$1$ matrix and the corresponding non-ambiguous forest.}\label{fig0}%
\end{figure}%
In \cite{Aval2} Aval et al.\ presented a bijection between non-ambiguous trees and the corresponding subclass of Callan permutations using associated ordered trees as intermediate combinatorial objects.
As an easy corollary of the proof of Theorem~\ref{tetel1}, we deduce the number of non-ambiguous forests with characteristic matrix of size $n\times k$ in Section~\ref{sec2}.
A special class of non-ambiguous trees has interesting connection to the $J$-Bessel function.
Let denote $b_n$ the number of complete non-ambiguous trees with $n$ internal vertices (in OEIS \cite{OEIS} as sequence A002190). Then \cite{Aval1}:
\[\sum_{n\geq 0}b_n\frac{x^{n+1}}{(n+1)!^2}=-\ln J_0(2\sqrt{x}),\]
where $J_0$ is the Bessel function.
Aisbett \cite{Aisbett} and Jin \cite{Jin} show strong connections of complete non-ambiguous trees with a poset of vector partitions.\par
Let $\alpha=(a_1,\dots,a_n)$ and $\beta=(b_1,\dots,b_n)$ be permutations of $\{1,\dots,n\}$. We say the pair $(\alpha,\beta)$ has {\it common rise\/} at position $i$ (where $1\le i\le n-1$),
if $a_i<a_{i+1}$ and $b_i<b_{i+1}$ hold at the same time.
In \cite{Aval1} Aval et al.\ made a nice observation. As a corollary, they noted that the number of complete non-ambiguous forests with $n$ leaves
has the same generating function as the number of pairs of permutations of $\{1,\dots,n\}$ with no common rise \cite{Carlitz2}, denoting these quantities by $\tau(n)$ and $\omega(n)$, respectively, 
\[\sum_{n\geq 0}\tau(n)\frac{x^n}{n!^2}=\left(\sum_{n\geq 0}\frac{(-1)^nx^n}{n!^2}\right)^{-1}=\sum_{n\geq 0}\omega(n)\frac{x^n}{n!^2};\]
and asked for a bijective explanation of the equality $\tau(n)=\omega(n)$.
Jin described a bijection with the intermediate combinatorial objects of certain heaps \cite{Jin}.
We present a more direct bijection in Section~\ref{sec3}, based on a labeled forest structure.
This is our second main result.
\begin{theorem}\label{tetel2}
There exists a bijection between the set of complete non-ambiguous forests with $n$ leaves and the set of pairs of permutations of $\{1,\dots,n\}$ with no common rise.
Thus these sets are equinumeruous.
\end{theorem}
In fact, we prove this theorem in a stronger form; we construct a bijection from the set of complete non-ambiguous forests with a fixed set of leaves to the corresponding subset of pairs of permutations with no common rise;
see Theorem~\ref{tetel3} in Section~\ref{sec3}.
\section{The number of $\bold\Gamma$-free 0-1 matrices}\label{sec2}
In this section we present a proof of Theorem~\ref{tetel1}. But we need to cite a folklore result first.\par
By {\it rooted forest\/} we mean a vertex-disjoint union of (unordered) rooted trees. Two rooted forests with the same vertex set are considered the same,
if and only if they have the same set of roots and they have the same edge set.
In our proofs, the edges will usually be directed from parent to child (so the components become arborescences).
Fix a totally ordered set $(V,<)$. We say that a rooted forest $F$ on vertex set $V$ is {\it increasing\/}, if whenever the vertex $u$ is the parent of vertex $v$ in $F$,
then $u<v$. (See Figure~\ref{fig1} for an example.) We note that in our figures we always list the children of a given parent in decreasing order
from left to right (their order ``does not count''),
the tree components are listed in the decreasing order of their roots, and we follow this order in our algorithms, too.
\begin{figure}[h]%
\includegraphics[scale=1.4]{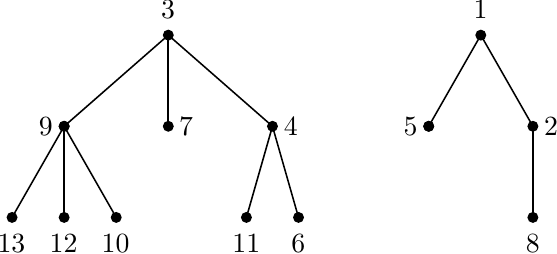}%
\centering
\caption{An increasing forest
on vertex set $\{1,\dots,13\}$.}\label{fig1}%
\end{figure}%
\begin{notation}\label{jeloles}
If $v$ is a vertex of a rooted forest $F$, then $F[v]$ denotes the rooted subtree of $F$ spanned by $v$ and its descendants (children, grandchildren etc.). The root of $F[v]$ is $v$.
\end{notation}
The following lemma is well known \cite{StanEc2}:
\begin{lemma}\label{monoton_erdo}
Let $V$ be a finite, totally ordered set. There exists a bijection $\pi$ from the set of increasing forests on vertex set $V$
to the set of permutations of $V$.
\end{lemma}
\begin{proof}
We can obviously assume, that $V=\{1,\dots,n\}$ for some $n$, equipped with the natural order.
Let $F$ be an increasing forest on $V$.
We will need the \emph{pre-order transversal} of a rooted tree $T$, denoted by $\alpha(T)$. It is a permutation of the vertices of $T$ which can be defined recursively as follows:
Let $r$ be the root of $T$. If $T$ has only one vertex, $r$, then $\alpha(T):=r$; otherwise let $v_1,\dots,v_m$ be the children of $r$ in decreasing order,
and set $\alpha(T):=r\alpha(T[v_1])\dots\alpha(T[v_m])$.
(The ``product'' means concatenation here.)
Now, if the (rooted tree) components of $F$ are $T_1,\dots,T_k$, listed in the decreasing order of their roots,
then $\pi(F)$ is defined to be the permutation $\alpha(T_1)\alpha(T_2)\dots\alpha(T_k)$.
For example, $\pi$ maps the increasing forest in Figure~\ref{fig1} to the permutation $(3,9,13,12,10,7,4,11,6,1,5,2,8)$, for $n=13$.
It is straightforward to check that $\pi$ is a bijection, as the lemma states. For further details (with a slightly different terminology), see Section~1.5 in \cite{StanEc2}.
\end{proof}
\begin{corollary}{\rm\cite{StanEc2}}
The number of increasing forests on vertex set $\{1,\dots,n\}$ is $n!$.
\end{corollary}
Now we are ready to prove our first main theorem.
\begin{proof}[\textbf{\textit{Proof of Theorem~\ref{tetel1}.}}]
Let $\mathcal M$ denote the set of $\Gamma$-free $n\times k$ \zo matrices, and let $\mathcal S$ denote the set of $(n,k)$-Callan sequences.
Now we define a function $\phi\colon\mathcal M\to\mathcal S$, which will be proved to be bijective.
The reader is advised to follow the construction on Figure~\ref{fig2}.\par
Let $M$ be an arbitrary matrix from $\mathcal M$. The rows of $M$ can be identified with the numbers $1,\dots,n$,
and the columns can be identified with the numbers $1,\dots,k$, and we index rows and columns as described in Convention~\ref{konv}.\par
We say that an element $1$ in $M$ is a {\it top-$1$\/}, if it is the highest 1 in its column, i.e.\
if there is no 1 above it in its column.
There are three types of rows in $M$: There are the all-$0$ rows; there are the rows that contain at least one top-$1$, we call them {\it top rows\/};
and there are the rows that contain at least one $1$ but none of their $1$'s is top-$1$, we call them {\it special rows\/}.
In order to reduce the amount of required formalism, we define $\phi(M)$ with the help of auxiliary structures. Let $G=G_M$ be the following directed graph:
The vertices of $G$ are the (positions of) $1$'s of $M$. For each non-top $1$ vertex $u$, we add a directed edge starting from $u$
and ending at the next (lowest) $1$ above $u$ in its column; and there are no other edges in $G$.
If the edge $e$ starts from row $s$ and ends in row $t$, we say that the {\it length\/} of $e$ is $t-s$.
Since $M$ is $\Gamma$-free, the edges starting from an arbitrary fixed row have pairwise distinct lengths.
For each special row $s$, the longest edge starting from $s$ is called {\it special\/}. We note that this is a valid definition
and we underline that no special edge starts from a top row. The non-special edges of $G$ are called {\it regular\/}. The next step is to ``project the special edges horizontally'':
Let $H=H_M$ be the directed graph whose vertices are the {\it non-all-$0$ rows\/} (row indices) of $M$, and for each {\it special\/} edge $e$ of $G$,
there is an edge $e'$ in $H$ such that if $e$ starts from row $s$ and ends in row $t$, then $e'$ starts from vertex $s$ and ends at vertex $t$;
and there are no other edges in $H$. $H$ has a very simple structure. Since all edges are directed ``upwards'', there is no directed cycle in $H$.
All vertices corresponding to a special row have outdegree $1$, and all vertices corresponding to a top row have outdegree $0$.
In each row of $M$, only the rightmost $1$ can be the end vertex of an edge in $G$, otherwise a $\Gamma$ would be formed;
and for this unique possible end vertex $v$, there is at most one edge in $G$ which ends at $v$,
the edge that starts from highest $1$ below $v$ in its column (if such a $1$ exist).
This implies that every vertex of $H$ has indegree at most $1$. These altogether mean that $H$
consists of vertex-disjoint directed paths. The end vertices of the path components correspond to a top row, while the other vertices of $H$ correspond to special rows.
Consequently, the number of components of $H$ is equal to the number of top rows in $M$, this number is denoted by $m$.
We assign a pair $(R,C)$ to each component $P$ of $H$ where $R$ is the set of vertices (row indices) of $P$,
and $C$ is the set of column indices of top-$1$'s in the top row corresponding to the endpoint of $P$.\par
After doing this for all components, we get a collection of pairs $(R_1,C_1),\dots,(R_m,C_m)$; we are left to define how to permute them to obtain a sequence.
The obtained sequence will be an $(n,k)$-Callan sequence, because it has the required properties:
The $R_i$'s are obviously pairwise disjoint nonempty subsets of $\{1,\dots,n\}$ by construction;
the $C_i$'s are nonempty subsets of $\{1,\dots,k\}$ because every top row contains at least one top-$1$,
and the $C_i$'s are pairwise disjoint because every column of $M$ contains at most one top-$1$ (all-$0$ columns have no top-$1$'s, the other columns have exactly one top-$1$'s).\par
Now we ``project the regular edges of $G$ horizontally''. To this end, we define a new directed graph $F=F_M$ as follows:
The vertices of $F$ are the pairs $(R_1,C_1),\dots,(R_m,C_m)$; and for each {\it regular\/} edge $e$ of $G$, there is an edge $\widetilde e$ in $F$,
such that if $e$ starts from row $s$ and ends in row $t$, then $\widetilde e$ is from $(R_i,C_i)$ to $(R_j,C_j)$ where $i$ and $j$ are the unique indices for which $s$ is contained in $R_i$ and $t$ is contained in $R_j$;
and there are no other edges in $F$. It turns out that $F$ is a rooted forest.
We have seen above that the $\Gamma$-free property implies that for each row $r$ of $M$, there is at most one edge in $G$ that ends in $r$.
This means that if $e$ is a regular edge in $G$ that starts from row $s$ and ends in row $t$, then $t$ is the smallest row index (lowest row)
in the set $R_j$ containing $t$, because there is no special edge ending in $t$. We also note that if $s$ is contained in the set $R_i$, then
$\min R_i \le s < t = \min R_j$.
So we can conclude that for each set $R_j$, there is at most
one regular edge in $G$ that ends in a row of $R_j$ (this row can only be the lowest row, and there is at most one edge that ends in that row),
implying that every vertex $(R_i,C_i)$ in $F$ has indegree at most $1$. The previous discussion also showed that if $\widetilde e$
is an edge from $(R_i,C_i)$ to $(R_j,C_j)$ in $F$, then $\min R_i<\min R_j$. From these we can see that $F$ is indeed a rooted forest (the edges are directed from parent to child),
and moreover, $F$ is an increasing forest on the set $\{(R_1,C_1),\dots,(R_m,C_m)\}$ equipped with the following total order $\prec$:
$$(R_i,C_i)\prec(R_j,C_j)\overset{\text{def}}{\iff}\min R_i<\min R_j.$$
(This order is total, because the $R_i$'s are pairwise disjoint, so their smallest elements are pairwise distinct.)
The bijection $\pi$ of Lemma~\ref{monoton_erdo} constructs a permutation $\pi(F)$ of the vertices $(R_1,C_1),\dots,(R_m,C_m)$.
Finally, we set $\phi(M):=\pi(F)$. We have already checked in the previous paragraph that $\phi(M)\in\mathcal S$, so the definition of $\phi$ is valid.
The construction is illustrated on Figure~\ref{fig2}, where the top-$1$'s and the special edges are colored red.
\par
\begin{figure}[h]%
\includegraphics[width=\hsize]{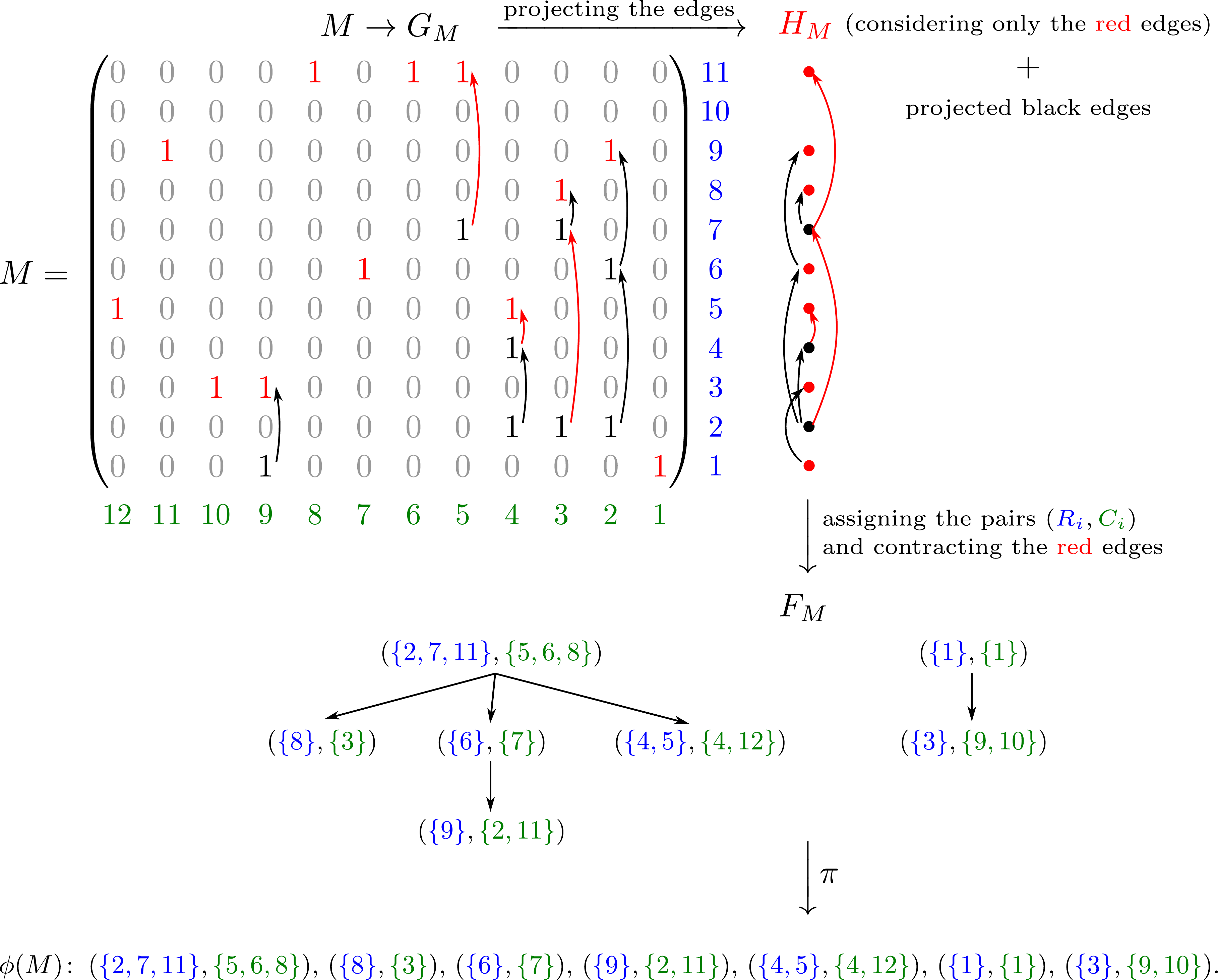}%
\centering
\caption{Illustration of $\phi$ (for $n=11$, $k=12$).}\label{fig2}%
\end{figure}%
Before proving the bijectivity of $\phi$, we summarize some properties of the construction:
\begin{enumerate}[label=(\roman*)]
\item\label{tul1} The rows in $\cup_{i=1}^m R_i$ are exactly those rows of $M$ that contain at least one $1$,
the columns in $\cup_{i=1}^m C_i$ are exactly those columns of $M$ that contain at least one $1$.
\item\label{tul2} In the row $\max R_i$ (from bottom), the top-$1$'s are exactly in the columns of $C_i$, for $i=1,\dots,m$. There are no top-$1$'s in other rows of $M$.
\item\label{tul3} The edges of $G_M$ are bijectively associated to the non-top $1$'s of $M$ (an edge corresponds to the non-top $1$ from which it starts).
\item\label{tul4} If there is an edge $e$ from row $s$ to row $t$ in $G_M$, then the start and end vertex of $e$ are in the column of the rightmost $1$ of the row $t$.
\item\label{tul5} The number of edges in $H_M$ is equal to the number of special edges in $G_M$.
(The ``projection'' $e\mapsto e'$ in the definition of $H_M$ is a bijection between the two edge sets.)
\item\label{tul6} The number of edges in $F_M$ is equal to the number of regular edges in $G_M$.
(The ``projection'' $e\mapsto \widetilde e$ in the definition of $F_M$ is a bijection between the two edge sets.)
\item\label{tul7} If $(R_j,C_j)$ is a child of $(R_i,C_i)$ in $F_M$, i.e.\ if there is an edge $\widetilde e$ from $(R_i,C_i)$ to $(R_j,C_j)$ in $F_M$,
then the regular edge $e$ of $G_M$ that corresponds to $\widetilde e$ ends in row $\min R_j$ and starts
in the highest row of $R_i$ that is lower than $\min R_j$. (Such a row exists, because $\min R_i < \min R_j$ by the increasing property of $F_M$.)
\end{enumerate}
The injectivity of the projection in \ref{tul6} was implicitly proved when we saw that every vertex in $F_M$ has indegree at most $1$.
The last statement of \ref{tul7} follows from the facts that $e$ cannot start from a higher row of $R_i$, because it is directed ``upwards'',
and $e$ cannot start from a lower row of $R_i$, because otherwise the regular edge $e$ would be longer than the special edge starting from the same row.
\par
Let $a=((R^*_1,C^*_1),\dots,(R^*_l,C^*_l))$ be an arbitrary $(n,k)$-Callan sequence.
We have to show that there is exactly one $\Gamma$-free $n\times k$ $0$-$1$ matrix $M^*$ for which $\phi(M^*)=a$.\par
So we reconstruct $M^*$ from $a$. The last step in constructing the image of a matrix was the application of the bijection $\pi$ from Lemma~\ref{monoton_erdo}.
This step is invertible, so we can view $a$ as $\pi^{-1}(a)$, an increasing forest on the set $\{(R^*_1,C^*_1),\dots,(R^*_l,C^*_l)\}$ with the total order $\prec$ defined above.
Set $F^*:=\pi^{-1}(a)$, and let $H^*$ be the directed graph that is the vertex-disjoint union of directed paths $P_i$
whose vertices are the elements of $R^*_i$ in increasing order along the path ($i=1,\dots,l$).
The point is that the top-$1$'s of $M^*$ are uniquely determined by property~\ref{tul2}; and we can uniquely determine the projection of the edges assigned to the non-top $1$'s,
using property~\ref{tul5} for special edges and using properties \ref{tul6}-\ref{tul7} for regular edges.
On a {\it projected edge\/} we mean a pair $(s,t)$ of rows, which can be thought as ``it starts from row $s$ and ends in row $t$'' ($s<t$). Further, we can figure out the original column of the projected edge $(s,t)$
once the row $t$ is completely filled, by property~\ref{tul4}. A row $r$ is completely filled if we know the positions of non-top $1$'s in that row (the top-$1$'s are already determined); and for that, we have to project back
the projected edges starting from $r$. For this reason, we may need first to fill the last rows of the children of $(R^*_i,C^*_i)$ in $F^*$, where $R^*_i\ni r$.
These lead to the following algorithm for constructing $M^*$:
\begin{itemize}
\item Start with an empty $n\times k$ matrix, then fill the rows of $\{1,\dots,n\}\setminus\cup_{i=1}^l R^*_i$ with $0$'s.
\item Order the vertices of $F^*$ in such a way that every non-root vertex precedes its parent (this can be done by post-order transversal, for example).
\item Process the vertices of $F^*$ in this order, and for vertex $(R^*_i,C^*_i)$, fill the rows of $R^*_i$ as follows ($i\in\{1,\dots,l\}$):
\begin{enumerate}[label=(\alph*)]
\item\label{step_a} In row $\max R^*_i$, place $1$'s into the columns of $C^*_i$.
\item\label{step_b} For each child $(R^*_j,C^*_j)$ of $(R^*_i,C^*_i)$: Let $c$ denote the column of the rightmost $1$ in row $\min R^*_j$ (that row is already filled),
let $r$ denote the highest row of $R^*_i$ that is lower than $\min R^*_j$ (cf.\ property~\ref{tul7}, and we note that such an $r$ exists by the increasing property of $F^*$), and place a $1$ into position $(r,c)$.
\item\label{step_c} Fill the empty positions in row $\max R^*_i$ with $0$'s.
\item\label{step_d} Let $r_h$ be the $h^{\text{th}}$ highest row (index) in $R^*_i$. For $h=2,3,\dots,|R^*_i|$, in row $r_h$ place a $1$ below the rightmost $1$ of row $r_{h-1}$,
and then fill the empty positions of $r_h$ with $0$'s.
\end{enumerate}
\end{itemize}\par
After the previous discussion, it should be clear that the obtained matrix $M^*$ is the unique candidate to be the inverse of $a$.
(We note that the outcome $M^*$ of the algorithm does not depend on the actual order chosen in second step. We place at least one $1$ to each row of $\cup_{i=1}^l R^*_i$,
so the ``rightmost 1'' always makes sense in the algorithm.) So we are left to show that $M^*$ is $\Gamma$-free and that $\phi(M^*)=a$.
In our task, the following observation will be helpful:
\begin{enumerate}[resume,label=(\roman*)]
\item\label{tul8} Let $(R^*_i,C^*_i)$ be an arbitrary vertex of $F^*$, and let $J\subseteq\{1,\dots,l\}$ denote the set of indices $j$ for which $(R^*_j,C^*_j)$ is a vertex of $F^*[(R^*_i,C^*_i)]$ (cf.\ Notation~\ref{jeloles}).
In $M^*$ every $1$ of a row of $R^*_i$ is contained in a column of $\cup_{j\in J}C^*_j$.
The rightmost $1$ of row $\min R^*_i$ is in the rightmost column of $\cup_{j\in J}C^*_j$.
\end{enumerate}
This can be proved by a routine induction on the height of $(R^*_i,C^*_i)$ in $F^*$.\par
It is crucial to see that in step~\ref{step_b} when we place a $1$ into position $(r,c)$ we project back the edge $(R^*_i,C^*_i)\to(R^*_j,C^*_j)$ of $F^*$ to a regular edge,
while in step~\ref{step_d} we project back the edges of $H^*$ to special edges.
In order to see that this is what really happens in the algorithm, we have to check that there are no other $1$'s between the positions $(r,c)$ and $(\min R^*_j,c)$
(the newly written $1$ and the ``rightmost $1$'') in the final $M^*$, otherwise the edge starting from $(r,c)$ would not end at $(\min R^*_j,c)$,
resulting $\phi(M^*)\ne a$. The analogous check for step~\ref{step_d} is also necessary.
So we take a closer look at the way the $1$'s are placed in steps \ref{step_a}-\ref{step_d}.
When step~\ref{step_a} is applied, the newly written $1$'s are the first $1$'s appearing in their column. This is because, by property~\ref{tul8},
only those rows can contain $1$ in a column of $C^*_i$ that belong to an ancestor of $(R^*_i,C^*_i)$, but the ancestors of $(R^*_i,C^*_i)$ have not
been considered yet. Then, applying property~\ref{tul8} to the children of $(R^*_i,C^*_i)$, we can see that the newly written $1$'s in step~\ref{step_b}
has pairwise distinct columns, and different from the columns of $C^*_i$, too. This also means that we place $|R^*_i|-1$ {\it new\/} $1$'s in step~\ref{step_d}.
For steps \ref{step_b} and \ref{step_d}, the key observation is the following:
\item{($*$)} {\it The ``rightmost $1$'' in the description of steps \ref{step_b} and \ref{step_d}
is the lowest $1$ in its column at the moment of placing the new $1$ below it, for all possible indices $i,j$.}\medbreak
\noindent Before proving this, we note that ($*$) and the observation made on step~\ref{step_a} imply that {\it the algorithm fills the $1$'s of each column from top to bottom\/}.
Now we prove ($*$) by induction on the progress of the algorithm.
Suppose, by contradiction, that ($*$) does not hold for step~\ref{step_b} for some $i,j$; and let $u$ denote the ``rightmost $1$'' in question, and let $v$ denote the ``newly written $1$''.
In other words, suppose that there is a $1$ below $u$, denoted by $w$, when we want to write $v$. We can assume that there are no other $1$'s between $u$ and $w$.
As seen before, the substeps of step~\ref{step_b} involve pairwise distinct columns $c$ for the children of this given $(R^*_i,C^*_i)$, and $c\notin C^*_i$,
so $v$ is the first $1$ in the column of $u$ that is written into a row of $R^*_i$. Since $u$ is in row $\min R^*_j$, $w$ is not in $R^*_j$ (and not in $R^*_i$, neither);
let $R^*_k$ be the set containing $w$ ($j\ne k\ne i$). By the induction hypothesis and an easy analysis of the algorithm,
we know that $w$ was placed below $u$ in a former step~\ref{step_b}, which means that
$(R^*_j,C^*_j)$ has two parents in $F^*$, $(R^*_i,C^*_i)$ and $(R^*_k,C^*_k)$, contradiction.
After processing the steps \ref{step_a}-\ref{step_c}, all the $1$'s written in a row of $R^*_i$
are the lowest ones in their column, as just seen. From this, it is very easy to conclude that ($*$) holds for step~\ref{step_d} as well.\par
We prove now that $M^*$ is $\Gamma$-free, i.e.\ $M^*\in\mathcal M$.
Suppose, by contradiction, that $M^*$ contains three $1$'s in a $\Gamma$-configuration,
and let $u,v$ and $w$ denote the upper-left, lower-left and the upper-right one, respectively: $\begin{smallmatrix}u&w\\v&\end{smallmatrix}$.
We can assume that $v$ is chosen so that there are no $1$'s between $u$ and $v$ in $M^*$.
We saw in the previous paragraph that the algorithm fills the $1$'s of each column from top to bottom in the way described in ($*$),
which means that $v$ was placed below $u$ in a step \ref{step_b} or \ref{step_d}. But this is a contradiction, because $u$ is not the rightmost $1$
of its row (and the whole row of $u$ has been already filled, when $v$ is placed).\par
Now we sketch why $\phi(M^*)=a$.
As we underlined earlier, for steps \ref{step_b} and \ref{step_d}, there are no other $1$'s between the ``rightmost $1$'' and the ``newly written $1$'' in the final $M^*$.
This is ensured by ($*$) at the moment after writing the new $1$, and it will not change later neither, because the algorithm fills the $1$'s from top to bottom per column.
From this it is easy to see that in step~\ref{step_b} we indeed project back the edges of $F^*$ to regular edges, more precisely,
we ensure that the regular edges (of $G_{M^*}$) starting from a row of $R^*_i$ become, in $F_{M^*}$, the edges of $F^*$ starting from $(R^*_i,C^*_i)$.
(We leave the reader to check that the newly placed $1$'s in step~\ref{step_b} are start vertices of {\it regular\/} edges in the final $M^*$.)
Similarly, in step~\ref{step_d} we indeed project back the edges of $H^*$ to special edges. (We leave the reader to check that the newly placed $1$'s in step~\ref{step_d} are start vertices of
{\it special\/} edges in the final $M^*$.)
Finally, we note that since the $1$'s are filled from top to bottom per column,
hence the $1$'s placed in step~\ref{step_a} will be top-$1$'s in the final $M^*$, too -- which is also required for $M^*$ being the inverse of $a$.
Some minor details were skipped, but it is easy to see now that $H_{M^*}=H^*$, the pairs assigned to $M^*$ are $\{(R^*_i,C^*_i)\}_{i=1}^l$, and $F_{M^*}=F^*$,
which lead to $\phi(M^*)=a$, as stated.
\end{proof}
Recall from Section~\ref{sec1} that non-ambiguous forests correspond to those $\Gamma$-free $0$-$1$ matrices that have no all-$0$ rows and columns.
So in order to count non-ambiguous forests, one needs to enumerate these $\Gamma$-free matrices. This was done in~\cite{BH2}, and it is also an easy corollary of our previous construction.
\begin{corollary}
The number of non-ambiguous forests with characteristic matrix of size $n\times k$ is
$$
\sum_{m=0}^{\min(n,k)}(m!)^2\stirling{n}{m}\stirling{k}{m}.
$$
\end{corollary}
\begin{proof}
By the definition of non-ambiguous forests, we have to count the $\Gamma$-free $n\times k$ $0$-$1$ matrices without all-$0$ rows and columns. Let $\mathcal N$ denote the set of these matrices.
In the proof of Theorem~\ref{tetel1}, $\phi$ establishes a bijective correspondence between $\Gamma$-free $n\times k$ $0$-$1$ matrices and $(n,k)$-Callan sequences.
Property~\ref{tul1} of $\phi$ shows that the restriction of $\phi$ to $\mathcal N$ is a bijection between $\mathcal N$ and the set of those $(n,k)$-Callan sequences $(R_1,C_1),\dots,(R_m,C_m)$
for which $\{R_1,\dots,R_m\}$ is a partition of $\{1,\dots,n\}$ and 
$\{C_1,\dots,C_m\}$ is a partition of $\{1,\dots,k\}$. These sequences are clearly counted by the sum in the statement.
\end{proof}
\section{The generating function of $\bold\Gamma$-free matrices}\label{sec2.5}
We can use the bijection to derive easily the generating function of $\Gamma$-free matrices. We introduce some notations. Let $\mathcal{G}$ denote the set of $\Gamma$-free matrices.  For $m\in{\mathcal G}$ let $r(m)$ denote the number of rows, $c(m)$ the number of columns, $r_e(m)$ the number of empty rows, $c_e(m)$ the number of empty columns, and $r_t(m)$ denote the number of top rows, rows that contain top $1$'s. 
Let $M(x,y,a,b,t)$ be the generating function of $\Gamma$-free matrices defined as follows:
\[M(x,y,a,b,t)=\sum_{m\in \mathcal{G}}t^{r_t(m)}a^{r_e(m)}b^{c_e(m)}\frac{x^{r(m)}}{r(m)!}\frac{y^{c(m)}}{c(m)!}.\]
We have the next theorem.
\begin{corollary}
The generating function of $\Gamma$-free matrices is
\begin{align*}
M(x,y,a,b,t)=\frac{e^{ax}e^{by}}{1-t(e^x-1)(e^y-1)}.
\end{align*}
\end{corollary}
\begin{proof}
We have seen that there is a bijection between $\Gamma$-free matrices of size $n\times k$ and $(n,k)$-Callan-sequences. Moreover, we have seen that $\cup_{i=1}^mR_i$ (resp.\ $\cup_{i=1}^mC_i$) corresponds to the set of non-empty rows (resp.\ columns), and each top row corresponds to a pair $(R_i,C_i)$.
We can construct an $(n,k)$-Callan sequence as a sequence (permutation) of pairs of non-empty sets from $\{1,\ldots, n\}$ and $\{1,\ldots,k\}$ with two further sets: the set of remaining elements from $\{1,\ldots,n\}$ that may be empty; and the set of remaining elements from $\{1,\ldots,k\}$ that may also be empty. We follow the standard technique of symbolic method and use the usual notations of \cite{Flajolet} for combinatorial constructions without explicitly defining these here. 
The construction for Callan sequences described above is symbolically formulated as follows:
\[(\mbox{SEQ}(\gamma\mbox{SET}_{>0}(\mathcal{X})\times \mbox{SET}_{>0}(\mathcal{Y})))\times(\mbox{SET}(\alpha \mathcal{X}))\times(\mbox{SET}(\beta \mathcal{Y})),\]
where $\mathcal{X}$ is the atomic class from $\{1,\ldots, n\}$,  $\mathcal{Y}$ is the atomic class from $\{1,\ldots,k\}$, $\alpha$ marks the empty rows, $\beta$ marks the empty columns, and $\gamma$ marks the top rows. Based on the correspondence: $\mathcal{X}\rightarrow x$, $\mathcal{Y}\rightarrow y$, $\alpha\rightarrow a$, $\beta\rightarrow b$, and $\gamma\rightarrow t$ the construction translates to the formula given in the theorem.
\end{proof}

\section{Proof of Theorem~\ref{tetel2}}\label{sec3}
This section is devoted to the proof of Theorem~\ref{tetel2}. Our construction is based on the fact that pairs of permutations of $\{1,\dots,n\}$ with no common rise
can be encoded with certain labeled rooted forests. We present this as a lemma, but we think it is interesting in its own right.\par
The set of permutations of $\{1,\dots,n\}$ is denoted by $S_n$; and for a permutation $\eta=(\eta_1,\dots,\eta_n)\in S_n$,
we denote by $P_{\eta}$ the set $\{(1,\eta_1),\dots,(n,\eta_n)\}$.\par
Recall Convention~\ref{konv} and Figure~\ref{fig0} where the roots are colored red: The children of an ``element $1$'' (vertex $u$) in a non-ambiguous forest are the 
lowest $1$ above $u$ (if such a $1$ exist) and the rightmost $1$ on the left side of $u$ (if such a $1$ exists).
We often mix the (characteristic) matrix terminology and the graph terminology for non-ambiguous forests as in the previous sentence, if it does not lead to confusion.
\par
Let $M=\chi_A$ be the characteristic matrix of a non-ambiguous forest $A$.
Analogously to the notion of top-$1$'s introduced in Section~\ref{sec2}, we say that an element $1$ in $M$ is a \mbox{\emph{leading-$1$}},
if it is the leftmost $1$ in its row.
Clearly, an element $1$ in $M$ has exactly one child in $A$ if and only if the element is either a top-$1$ but not a leading-$1$ or
it is a leading-$1$ but not a top-$1$.
So $A$ is complete if and only if there is no such element, i.e.\ the set of top-$1$'s coincides with the set of leading-$1$'s in $M$.
Since $M$ has no all-$0$ rows and columns, there is exactly one leading-$1$ in each row and there is exactly one top-$1$ in each column.
So the top-$1$'s can coincide with the leading-$1$'s only if $M$ is a square matrix.
Moreover, if $A$ is complete, then the leaves are exactly the top-$1$'s / leading-$1$'s,
because a leaf is always a top-$1$, and now each top-$1$ is a leaf because it is also a leading-$1$.
This means that in case of complete $A$,
every row and every column contains exactly one leaf (the leading-$1$ or top-$1$), and that the number of rows / columns
is equal to the number of leaves. We have just proved the following.
\begin{lemma}\label{complete}
The characteristic matrices of \emph{complete} non-ambiguous forests with $n$ leaves are
exactly those $\Gamma$-free $0$-$1$ matrices without all-$0$ rows and columns
that have size $n\times n$ and in which every top-$1$ is a leading-$1$.
In such a matrix $M$ the leaves are exactly the top-$1$'s,
and the set of positions of leaves / top-$1$'s is $P_{\eta}$ for some $\eta\in S_n$.
\end{lemma}\par
A pair $(\alpha,\beta)\in S_n\times S_n$ will always be identified with the sequence
$(a_1,b_1),\dots,(a_n,b_n)$, where $\alpha=(a_1,\dots,a_n)$ and $\beta=(b_1,\dots,b_n)$. If we want to emphasize this,
we write $(\alpha,\beta)^T$ for denoting the sequence.
In this setting, the pair $(\alpha,\beta)$ has {\it common rise\/} if there are two consecutive elements of $(\alpha,\beta)^T$ for which $(a_i,b_i)<(a_{i+1},b_{i+1})$,
where `$<$' stands for ``less in both coordinates''.
For an arbitrary $(\alpha,\beta)\in S_n\times S_n$, the set of elements of $(\alpha,\beta)^T$ is clearly equal to a set $P_{\eta}$ for some $\eta\in S_n$.\par
Now fix an $n$. We group the complete non-ambiguous forests with $n$ leaves by the set of their leaves (i.e.\ the set of positions $(i,j)$ of the leaves)
and group the members of $S_n\times S_n$ without common rise by the set of their elements -- considering the members of $S_n\times S_n$ as $n$-element sequences.
From the discussion above, each of these groups can be described by a permutation of $S_n$.
Hence, Theorem~\ref{tetel2} clearly follows from the following theorem:
\begin{theorem}\label{tetel3}
Fix an arbitrary $\eta\in S_n$. Let $\mathcal C$ be the set of those complete non-ambiguous forests (with $n$ leaves) whose set of leaves is $P_{\eta}$.
Let $\mathcal P$ be the set of those pairs $(\alpha,\beta)\in S_n\times S_n$ that have no common rise and for which the set of elements of $(\alpha,\beta)^T$ is $P_{\eta}$,
i.e.\ $\mathcal P$ is the set of permutations of $P_{\eta}$ with no common rise (in the sense above).
Then there exists a bijection between $\mathcal C$ and $\mathcal P$, thus $|\mathcal C|=|\mathcal P|$.
\end{theorem}%
\begin{figure}[h]%
\includegraphics[scale=0.4]{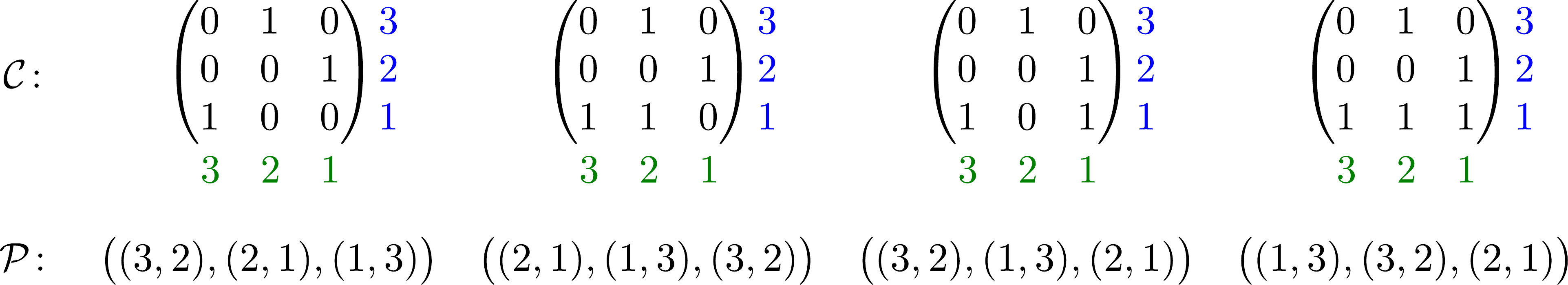}%
\centering
\caption{The sets $\mathcal C$ and $\mathcal P$ for $n=3$, $\eta=(3,1,2)$.}%
\end{figure}%
\begin{proof}
Let $\eta=(\eta_1,\dots,\eta_n)$. 
We think the elements of $\mathcal C$ as $n\times n$ $0$-$1$ matrices, as characterized in Lemma~\ref{complete}.
Recall the construction of the proof of Theorem~\ref{tetel1}.\par
As a first step, we apply the injective map $M\mapsto F_M=\pi^{-1}(\phi(M))$ to the matrices $M\in\mathcal C$, where $F_M$, $\pi$, and $\phi$ are defined in Section~\ref{sec2}.
This map establishes a bijective correspondence between the matrices of $\mathcal C$ and the increasing forests in the set $\widetilde{\mathcal C}:=\{F_M:M\in\mathcal C\}$.
Now we are going to give a simple description of $\widetilde{\mathcal C}$.\par
By Lemma~\ref{complete}, we know that in every matrix of $\mathcal C$, the set of top-$1$'s is $P_{\eta}$.
Let $\mathcal D$ denote the set of $\Gamma$-free $n\times n$ $0$-$1$ matrices in which the set of top-$1$'s is $P_{\eta}$.
Clearly, $\mathcal D\supseteq\mathcal C$, but not every matrix of $\mathcal D$ corresponds to a {\it complete\/} non-ambiguous forest.
Now we pick an arbitrary matrix $M\in\mathcal D$, and examine the increasing forest $F_M$.
We know that each row of $M$ is a top row: the row $i$ has exactly one top-$1$, which is in column $\eta_i$, by the definition of $\mathcal D$.
This means that the vertex set of $F_M$ is $P_{\eta}$. (More precisely, the vertices of $F_M$ have the form $(\{i\},\{\eta_i\})$, but we can leave the braces.)
As there are no special edges in $G_M$ (see Section~\ref{sec2}), every non-top $1$ of $M$ contributes to the edges of $F_M$.
In this special case, $F_M$ can be obtained from $M$ as follows: 
View $M$ as a non-ambiguous forest (with vertices and edges), project it horizontally, then label the vertex corresponding to the projected row $i$ with $(i,\eta_i)$, for $i=1,\dots,n$,
and orient the edges upwards. It should be clear from the proof of Theorem~\ref{tetel1} that the map $M\mapsto F_M$ is a bijection between $\mathcal D$
and the set of increasing forests on vertex set $P_{\eta}$. We denote the latter set by $\widetilde{\mathcal D}$. The total order $\prec$ on vertex set $P_{\eta}$ is simply the order by first coordinate,
i.e.\ now the condition ``increasing'' means that every child must have greater first coordinate than its parent has.\par
We know that $\widetilde{\mathcal C}\subseteq\widetilde{\mathcal D}$. An increasing forest $F^*\in\widetilde{\mathcal D}$ is in $\widetilde{\mathcal C}$ if and only if
the corresponding matrix $M^*\in\mathcal D$ (for which $F_{M^*}=F^*$) is in $\mathcal C$.
By Lemma~\ref{complete}, the matrix $M^*$ of $\mathcal D$ is in $\mathcal C$ if and only if every top-$1$ is a leading-$1$ in $M^*$, in other words, iff in every row of $M^*$ the non-top $1$'s are on the right side of the top-$1$ in that row.
The column index of a non-top $1$ can be read off from $F^*$.
In the proof of Theorem~\ref{tetel1}, step~\ref{step_b} of the inverse construction describes how the non-top $1$'s are placed to $M^*$ (step~\ref{step_d} does not place any $1$'s for $F^*\in\widetilde{\mathcal D}$):
For an arbitrary row $r$, the non-top $1$'s are associated to the children of $(r,\eta_r)$ in $F^*$. Namely, for each child $(s,\eta_s)$ of $(r,\eta_r)$, a $1$ is placed to row~$r$ into the column of the rightmost $1$ of row $s$,
and that column is the minimum of the second coordinates of the vertices in the subtree $F^*[(s,\eta_s)]$, applying the last statement of property~\ref{tul8} to the child $(s,\eta_s)$.
(Cf.\ Notation~\ref{jeloles} and Convention~\ref{konv}.) We conclude that the increasing forest $F^*\in\widetilde{\mathcal D}$ is in $\widetilde{\mathcal C}$
if and only if for every vertex $(a,b)$ of $F^*$ and for every child $(c,d)$ of $(a,b)$, there exists a vertex of $F^*[(c,d)]$ with second coordinate less than $b$.\par
We introduce a term for describing the forests of $\widetilde{\mathcal C}$. We say that $F$ is a {\it properly labeled forest\/} on vertex set $P_{\eta}$ (for some $\eta\in S_n$),
if $F$ is an (unordered) rooted forest on vertex set $P_{\eta}$, satisfying that whenever $(a,b)$ is the parent of $(c,d)$ in $F$, then
$a<c$ and there exists a vertex $(x,y)\in F[(c,d)]$ such that $b>y$. (See Figure~\ref{fig5} for an example.)
We can summarize the above investigations as $\widetilde{\mathcal C}$ is the set
of properly labeled forests on $P_{\eta}$. So it is enough to prove the following lemma.
\end{proof}
\begin{figure}[h]%
\includegraphics[width=\hsize]{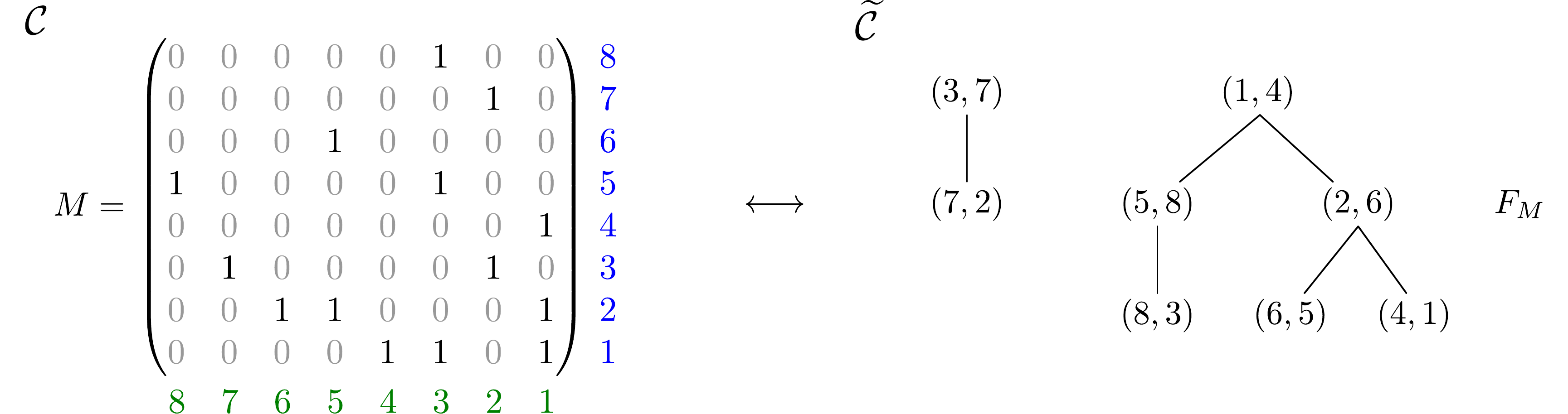}%
\centering
\caption{A complete non-ambiguous forest and the corresponding properly labeled forest.}\label{fig5}%
\end{figure}%
\begin{lemma}
For a given $\eta\in S_n$, let $\widetilde{\mathcal C}$ denote the properly labeled forests on vertex set $P_{\eta}$,
and let $\mathcal P$ denote the set of permutations of $P_{\eta}$ with no common rise.
Then there exists a bijection between $\widetilde{\mathcal C}$ and $\mathcal P$, thus $|\widetilde{\mathcal C}|=|\mathcal P|$.
\end{lemma}
\begin{proof}
We begin with some conventions. For $(a,b),(c,d)\in P_{\eta}$, let
\begin{align*}
(a,b)<_1(c,d)&\overset{\text{def}}{\iff}a<c;\\
(a,b)<_2(c,d)&\overset{\text{def}}{\iff}b<d.
\end{align*}%
\def\wtc#1{\widetilde{\mathcal #1}}%
In this proof we will work with increasing forests on vertex set $(P_{\eta},<_1)$. We follow the conventions introduced at the beginning of Section~\ref{sec2}
(cf.\ Figure~\ref{fig1}):
We always list the children of a given parent of a forest $F$ in decreasing order from left to right with respect to the order $<_{1}$, and the phrase ``{\it leftmost/first\/} child''
refers to this list (i.e.\ it means the child of a given parent with biggest first coordinate). The tree components of $F$ are also listed in the decreasing order of their roots with respect to $<_{1}$.\par
As a first step, we apply $\pi^{-1}$ to the permutations of $(P_{\eta},<_1)$, where $\pi$ is the bijection introduced in the proof of Lemma~\ref{monoton_erdo}.
(So we consider only the first coordinates with the natural order when building the forest structure from a permutation, see also Figure~\ref{fig4}.)
By the lemma, $\pi^{-1}$ establishes a bijective correspondence between the permutations of $(P_{\eta},<_1)$ and the increasing forests on $(P_{\eta},<_1)$.
It is known, or an easy analysis of $\pi$ shows, that for any permutation (sequence) $S=(\bold{s}_1,\dots,\bold{s}_n)$ of $P_{\eta}$ and any $i\in\{1,\dots,n-1\}$,
the inequality $\bold s_i<_1\bold s_{i+1}$ holds if and only if $\bold s_{i+1}$ is the {\it leftmost\/} child of $\bold s_i$ in the increasing forest $\pi^{-1}(S)$.
This implies that $S$ has no common rise if and only if, for every vertex $u$ of $\pi^{-1}(S)$, the leftmost child of $u$ (if $u$ is a non-leaf)
has smaller second coordinate than $u$ has. So $\wtc P:=\{\pi^{-1}(S):S\in\mathcal P\}$
is the set of those (unordered) rooted forests $F$ on vertex set $P_{\eta}$ which satisfy the following conditions:
\begin{itemize}[noitemsep]
\item[(c1)] whenever vertex $v$ is a child of vertex $u$ in $F$, then $u<_1v$, i.e.\ $F$ is increasing with respect to $<_1$;
\item[(c2)] and whenever vertex $v$ is the {\it leftmost\/} child of vertex $u$ in $F$, then $v<_2u$.
\end{itemize}
We say that a forest $F$ is \emph{leftmost-valid}, if it satisfies conditions (c1)-(c2). So $\wtc P$ is the set of leftmost-valid forests on $P_{\eta}$.\par
In order to complete the proof, we give a bijection $\psi\colon\wtc P\to\wtc C$.
Both the notion of leftmost-valid forest and the notion of properly labeled forest can be extended for any vertex set $V\subseteq P_{\eta}$, without any modification.
First we define a bijective conversion function $f$ from the set of leftmost-valid \emph{trees} (one-component forests) to the set of properly labeled \emph{trees},
such that $f(T)$ has the same vertex set and root as $T$, for any leftmost-valid tree $T$.
Then we can define $\psi$. For an aribtrary $F\in\wtc P$, if $F$ has (the leftmost-valid tree) components $C_1,\dots,C_m$, then $\psi(F)$ is defined to be the vertex-disjoint union of the properly labeled trees $f(C_1),\dots,f(C_m)$.
As $f$ keeps the vertex set, $\psi(F)\in\wtc C$.\par
Now we give the (recursive) definition of $f$. We note that for any leftmost-valid (resp.\ properly labeled) tree $T$ on $V$,
$T[v]$ is clearly a leftmost-valid (resp.\ properly labeled) tree for any $v\in V$, cf.\ Notation~\ref{jeloles}.
It is recommended to follow the conversion of $C_2$ on Figure~\ref{fig4}.
For a leftmost-valid tree $T$ on vertex set $V\subseteq P_{\eta}$, we define $f(T)$ as follows.
\begin{itemize}[noitemsep]
\item If $T$ has one vertex, then $f(T):=T$.
\item Otherwise, let $r$ be the root of $T$, and let $v_1,\dots,v_k$ be the children of $r$ in $<_1$-decreasing order. Set $T_i:=f(T[v_i])$ for $i=1,\dots,k$,
and consider the \emph{sequence}
\begin{equation}\label{tree}
T_1,\,T_2,\,\dots,\,T_k.
\end{equation}
Find the smallest index $i\in\{1,\dots,k\}$, if such an $i$ exists, for which no vertex of $T_i$ has smaller second coordinate than $r$ has. (We say that $T_i$ is the leftmost \emph{bad} tree.)
We note that $i\ne1$, as it will be justified later. Then remove the elements $T_1,\dots,T_i$ from \eqref{tree}, and add a new first element $T'_1$,
where $T'_1$ is the rooted tree obtained from $T_1,\dots,T_i$ by joining the roots of $T_1,\dots,T_{i-1}$ (as new children) to the root of $T_i$ (as parent/root).
In this way we obtain a new sequence
\begin{equation}\label{tree2}
T'_1,\,T'_2,\,\dots,\,T'_{k-i+1};
\end{equation}
where $T'_{j}=T_{j+i-1}$ for $j\ge 2$. We call this process \emph{merging}.
Then do the same for \eqref{tree2}: find the leftmost bad tree, merge.
Then repeat this for the new sequence, and so on, stop when no such index $i$ (bad tree) was found. We note that the process terminates,
because the length of the sequence strictly decreases in each merging step ($i\ne1$).
We end up with a sequence $\widetilde T_1,\dots,\widetilde T_l$.
Finally, $f(T)$ is defined to be the tree with root $r$ that is obtained by joining the vertex $r$ (as new root) to the roots of $\widetilde T_1,\dots,\widetilde T_l$.
\end{itemize}
\begin{figure}[h]%
\includegraphics[width=\hsize]{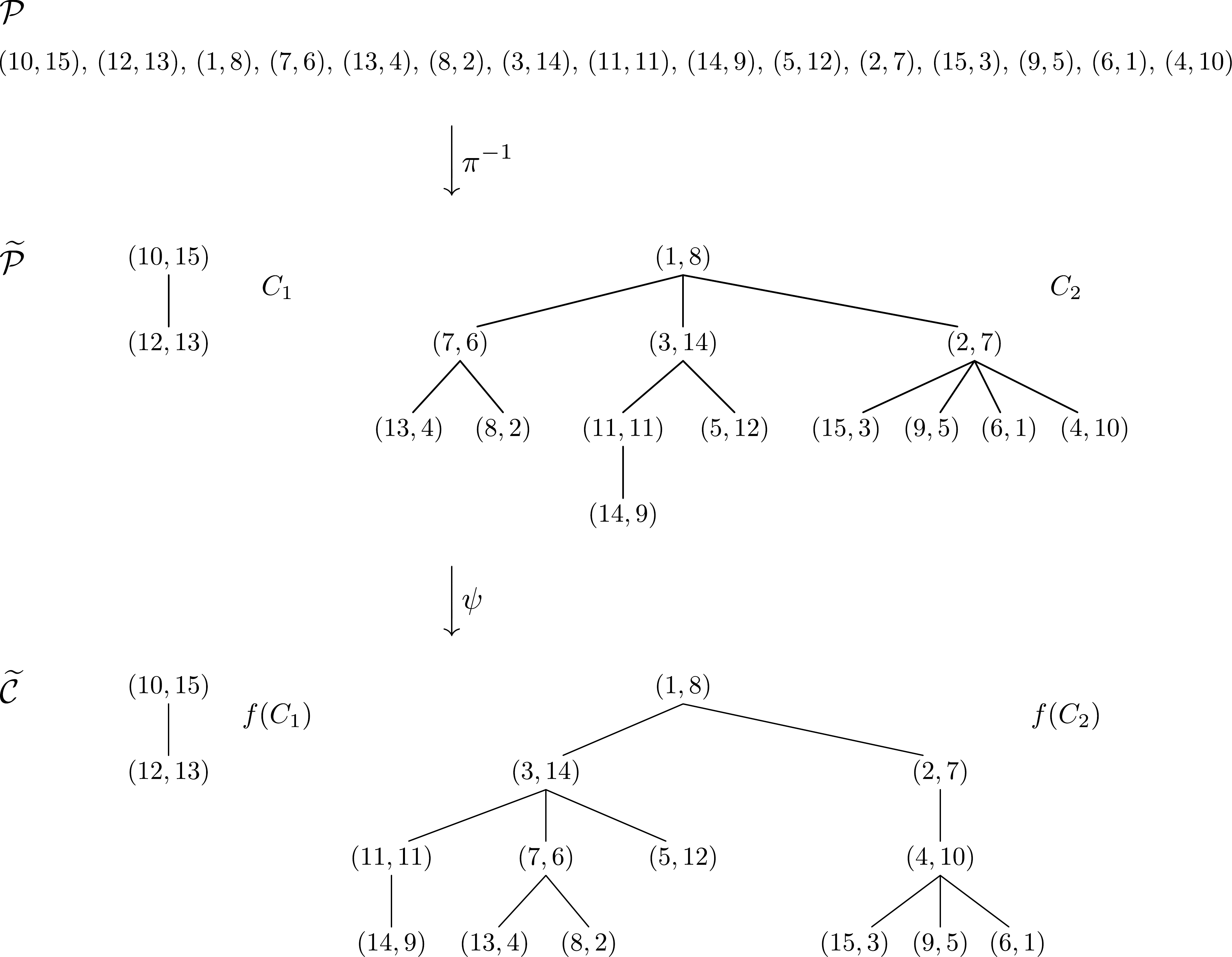}%
\centering
\caption{Illustration of
$\pi^{-1}|_{\mathcal P}$ and $\psi$ (for $n=15$).}\label{fig4}%
\end{figure}%
Now we justify why $i\ne1$, i.e.\ why the first tree in the actual tree sequence always has a vertex which is smaller than $r$ in the second coordinate.
This is true for the initial sequence \eqref{tree}, because $v_1$, the root of $T_1$, has smaller second coordinate than $r$ has, by the leftmost-validity
of $T$ and the fact that $f$ keeps the roots. And this vertex $v_1$ will keep staying in the first tree of the sequence after the mergings, too.\par
The fact that $f(T)$ has the same vertex set and root as $T$ can be verified by an easy induction.\par
Now we show why $f(T)$ is properly labeled. By induction, every tree of the initial sequence \eqref{tree} is properly labeled.
It is straightforward to see that this property is kept after each merging. We only have to check the new parent-child connections between the
root of the bad tree and its new children, the roots of the preceding trees. The monotonicity condition on the first coordinates is satisfied because the initial roots
$v_1,\dots,v_k$ are in $<_1$-decreasing order. The condition on the second coordinates is satisfied because 
if the root of the bad tree is $v_b$, then $r<_2 v_b$ (as $v_b$ is vertex of a bad tree), while in every preceding tree there exists a vertex $u$ with
$u<_2 r$ (as they are good trees), so a vertex $u$ with $u<_2 v_b$, as needed.
In the final stage we have the properly labeled good trees $\widetilde T_1,\dots,\widetilde T_l$, from which it is pretty obvious that $f(T)$ is properly labeled
(after checking the root $r$).\par
Now we sketch why $\psi$ is a bijection. It is clear that it is enough to show that for any fixed vertex set $V\subseteq P_{\eta}$,
the function $f$ (or more precisely, its restriction) is a bijection between the set of leftmost-valid trees on $V$ and the set of properly labeled trees on $V$.
This can be done by induction on the size of $V$.
Pick an arbitrary properly labeled tree $T^*$ on vertex set $V$, where $|V|\ge2$. Let $r$ be the root of $T^*$. The children of a given parent are listed in $<_1$-decreasing order
(the pre-order transversal follows this order in the next step).
\begin{itemize}[noitemsep]
\item[$\circ$] Find the first vertex $v_1$ in the pre-order transversal of $T^*$ (see the proof of Lemma~\ref{monoton_erdo}) for which $v_1<_2r$.
As $T^*$ is properly labeled, such a $v_1$ exists in $T^*[u]$, where $u$ is the leftmost child of $r$.
\item[$\circ$] We define a list (sequence) $\mathfrak L$ of subtrees. As an initial step, we define the first element of $\mathfrak L$ to be $T^*[v_1]$.
\item[$\circ$] Let $w_1,\dots,w_m$ be those siblings of $v_1$ in $T^*$, listed in $<_1$-decreasing order, for which $T^*[w_i]$ has a vertex with smaller second coordinate than the second coordinate of $r$.
We refer to these siblings as \emph{good} siblings. (The good siblings are on the right side of $v_1$, because $v_1$ was found by pre-order transversal.)
Add the trees $T^*[w_1],\dots,T^*[w_n]$ to the end of $\mathfrak L$ in this order. Then let $v_2$ be the parent of $v_1$ in $T^*$, and let $T^-_{v_2}$ denote the tree
obtained from $T^*[v_2]$ by deleting the subtrees $T^*[v_1]$, $T^*[w_1],\dots,T^*[w_n]$ from it. Add $T^-_{v_2}$ to the end of $\mathfrak L$.
Repeat this step for $v_2$ instead of $v_1$, and then for the parent $v_3$ of $v_2$, and so on, until we reach to the point when $v_i=u$.
At that point, all siblings of $v_i=u$ are good, due to the fact that $T^*$ is properly labeled. Add the subtrees $T^*[w]$ to $\mathfrak L$ for the siblings $w$ of $u$
from left to right as above, which finishes this step.
\item[$\circ$] We end up with a list $\mathfrak L$ of trees: $L_1,\dots,L_t$.
Finally, $\widehat T$ is defined to be the tree with root $r$ that is obtained by joining the vertex $r$ (as new root) to the roots of $f^{-1}(L_1),\dots,f^{-1}(L_t)$,
where the unique inverse images $f^{-1}(L_i)$ come from the induction hypothesis.
\end{itemize}
We claim that $\widehat T$ is the unique leftmost-valid tree on vertex set $V$ for which $f(\widehat T)=T^*$, proving the bijectivity of $f$.
The details are left to the reader.
\end{proof}

\end{document}